\documentclass[12pt, reqno]{amsart}
\usepackage{amssymb, amstext, amscd, amsmath}
\usepackage{color}

\makeatletter
\def\@cite#1#2{{\m@th\upshape\bfseries%
[{#1\if@tempswa{\m@th\upshape\mdseries, #2}\fi}]}} \makeatother
\theoremstyle{plain}
\newtheorem{thm}[subsection]{Theorem}
\newtheorem{cor}[subsection]{Corollary}
\newtheorem{prop}[subsection]{Proposition}
\newtheorem{lem}[subsection]{Lemma}
\theoremstyle{definition}
\newtheorem{rem}[subsection]{Remark}

\newtheorem{eg}[subsection]{Example}

\newcommand{\bC}{{\mathbb{C}}}

\newcommand{\bN}{{\mathbb{N}}}

  \newcommand{\A}{{\mathcal{A}}}
  \newcommand{\B}{{\mathcal{B}}}
  
  \newcommand{\D}{{\mathcal{D}}}

\renewcommand{\H}{{\mathcal{H}}}

  \newcommand{\K}{{\mathcal{K}}}

\newcommand{\M}{{\mathcal{M}}}

\renewcommand{\S}{{\mathcal{S}}}

\renewcommand{\phi}{\varphi}
\newcommand{\upchi}{{\raise.35ex\hbox{\ensuremath{\chi}}}}

\def\dim{{\rm dim}\,}

\begin{document}
\title[O.~A.~M.~on Hilbert Modules]
{Orthogonally Additive Mappings on Hilbert Modules}
\author[D. Ili\v{s}evi\'{c}, A. Turn\v{s}ek and D. Yang]
{Dijana Ili\v{s}evi\'c, Aleksej Turn\v{s}ek and Dilian Yang$^*$}
\address{Dijana Ili\v{s}evi\'c, Department of Mathematics,
University of Zagreb, Bijeni\v{c}ka 30,
10000 Zagreb, Croatia}
\email{ilisevic@math.hr}
\address{Aleksej Turn\v{s}ek, 
Faculty  of Maritime Studies and Transport\\
University of Ljubljana\\
Pot pomor\-\v{s}\v{c}akov 4\\
6320 Portoro\v{z}, Slovenia\\
and Institute of Mathematics, Physics and Mechanics\\
Jadranska 19\\
1000 Ljubljana, Slovenia
}
\email{aleksej.turnsek@fmf.uni-lj.si}
\address{Dilian Yang,
Department of Mathematics \& Statistics, University of Windsor,
ON N9B 3P4, Canada
}
\email{dyang@uwindsor.ca}

\begin{abstract}
In this paper, we study the representation of orthogonally additive mappings acting on Hilbert $C^*$-modules and Hilbert $H^*$-modules.
One of our main results shows that every continuous orthogonally additive mapping $f$ from a Hilbert module $W$ over $\K(\H)$ or $\H\S(\H)$
to a complex normed space is of the form
$f(x)=T(x)+\Phi(\langle x, x \rangle)$ for all $x\in W$,
where $T$ is a continuous additive mapping, and $\Phi$ is a continuous linear mapping.
\end{abstract}

\thanks{${}^*$ The first author was supported by the Ministry of Science, Educations and Sports of the
Republic of Croatia (Grant 037-0372784-2757).
The third author is partially supported by an NSERC Discovery Grant (Canada).}

\date{}
\maketitle

Let $\A$ be a $C^*$-algebra or an $H^*$-algebra, $(W, \langle \, . \, , . \, \rangle)$ be a Hilbert module over $\A$, and $G$ be a complex normed space.
A continuous mapping $W\to G$ is said to be orthogonally additive if for all $x, y \in W$,
\[
\langle x, y\rangle =0 \implies f(x+y)=f(x)+f(y).
\]
In this paper, we study the representation of orthogonally additive mappings. If $T:W\to G$ is a continuous additive mapping, and $\Phi: \A\to G$
is a continuous mapping, then clearly the mapping $f : W \to G$ defined by
\begin{align}
\label{E:rep}
f(x)=T(x)+\Phi(\langle x, x\rangle) \quad \text{for all} \quad x\in W
\end{align}
is a continuous orthogonally additive mapping. One of our main goals is to show that the converse also holds true if $\A$ is a $C^*$-algebra of compact operators or an $H^*$-algebra. In particular, this answers \cite[Problem 27]{z} affirmatively, not only for Hilbert $H^*$-modules, but for Hilbert $C^*$-modules over a $C^*$-algebra of compact operators as well.
Other related problems in \cite{z} have been also solved in \cite{i1, i2, i3}.


Orthogonally additive mappings have been extensively studied from many aspects.
See the survey \cite{Ratz} and the references therein for the representation of orthogonally additive mappings
on orthogonality spaces. Refer to \cite{Szab1, Szab2, Yang} for the connection between the existence
of even orthogonally additive mappings and inner product spaces.
Recently, several mathematicians have obtained some interesting results on orthogonally additive polynomials.
See, e.g.,  \cite{BLL, CLZ1, CLZ2, PPV}, among others.

The rest of this paper is organized as follows. In Section \ref{intro}, we give some necessary background and set up some notation.
Section \ref{general} deals with $\perp$-additive mappings on abelian groups. The context there is very general, so
the results there may be also useful in future. Applying the results of Section \ref{general}, we obtain the representation of orthogonally additive
mappings in general Hilbert modules in Section \ref{modules}.
In Section \ref{K(H)&HS(H)}, we
strengthen the results of Section \ref{modules} in the case when $\A$ is $\K(\H)$ or  $\H\S(\H)$.
The main result is then generalized to any Hilbert module over a $C^*$-algebra of compact operators or $H^*$-algebra in Section \ref{C*&H*}. In the last section, we obtain the representation of orthogonally additive
mappings on $\B(\H_1,\H_2)$.

\section{Preliminaries}\label{intro}

In this section, we give some necessary background and set up our notation.

\subsection{Hilbert $C^*$-modules and Hilbert $H^*$-modules}

Hilbert modules arise as generalizations of a complex Hilbert space when the complex field is replaced by a $C^*$-algebra or an $H^*$-algebra. The idea of replacing the complex numbers by the elements of a $C^*$-algebra first appeared in the work of Kaplansky \cite{k} and by the elements of a proper $H^*$-algebra in the work of Saworotnow \cite{sa}.

A $C^*$-algebra is a complex Banach $\ast$-algebra $(\A,\Vert \, .\,\Vert)$ such that $\Vert a^*a \Vert = \Vert a \Vert^2$ for all $a \in \A.$
An $H^*$-algebra is a complex Banach $\ast$-algebra $(\A,\Vert \, .\,\Vert)$, whose underlying Banach space is a Hilbert space with respect to the inner product
$\langle \, .\, ,\, .\,\rangle$ satisfying
$\langle ab,c \rangle=\langle b,a^*c \rangle$
and  $\langle ba,c \rangle=\langle b,ca^* \rangle$
for all  $a,b,c \in \A.$ The trace-class associated with an $H^*$-algebra $\A$ is defined as the set $\tau(\A)=\{ab : a,b \in \A\};$ it is a self-adjoint two-sided ideal of $\A$ which is dense in $\A.$

Some examples of $C^*$ algebras are $\B(\H)$ and $\K(\H),$ the algebras of all bounded operators, resp.~all compact operators, on some complex Hilbert space $\H.$ An example of an $H^*$-algebra is $\H\S(\H),$ the algebra of all Hilbert-Schmidt operators on $\H.$

An element $a$ in a $C^*$-algebra $\A$ is called positive $(a \geq 0)$ if it is self-adjoint and has nonnegative spectrum.
An element $a$ in an $H^*$-algebra $\A$ is called positive $(a \geq 0)$ if $\langle ax, x \rangle \geq 0$ for all $x \in \A.$
If $\A$ is a $C^*$-algebra (resp.~an $H^*$-algebra) then every positive $a \in \A$ (resp.~$a \in \tau(A)$) can be written as $a=b^*b$ for some $b \in \A.$
Every $a \in \A$ can be written as a linear combination of four positive elements, in both structures. A projection (i.e., self-adjoint idempotent) $e \in \A$ is called minimal if $e\A e=\mathbb{C}e.$

Let $\A$ be a $C^*$-algebra or an $H^*$-algebra. Let $W$ be an algebraic right $\A$-module which is a complex linear space with a compatible scalar multiplication, i.e.~$\lambda(xa)=(\lambda x)a=x(\lambda a)$ for all $x \in W,$ $a \in \A,$ $\lambda \in \mathbb{C}.$ The space $W$ is called a (right) inner product $\A$-module if there exists a generalized inner product, that is, a mapping
$\langle \, .\, ,\, .\, \rangle$ from $W \times W$ to $\A$ if $\A$ is a $C^*$-algebra, and to $\tau(A)$ if $\A$ is an $H^*$-algebra, having the following properties:
\begin{itemize}
\item[(i)] $\langle x,y+z \rangle = \langle x,y \rangle+\langle x,z\rangle$ for all $x,y,z \in W,$
\item[(ii)] $\langle x,ya \rangle=\langle x,y\rangle a$ for all $x,y \in W$ and $a \in \A,$
\item[(iii)] $\langle x,y\rangle^*=\langle y,x\rangle$ for all $x,y \in W,$
\item[(iv)] $\langle x,x\rangle \geq 0$ for all $x \in W,$ and $\langle x,x \rangle=0 \Leftrightarrow x=0.$
\end{itemize}

If $W$ is an inner product module over $(\A, \Vert \, . \, \Vert)$, then for $x \in W$ we write $\Vert x \Vert_W = \Vert \langle x,x \rangle \Vert^{\frac{1}{2}}.$
If $W$ is complete with respect to this norm, then it is called a Hilbert $\A$-module, or a Hilbert $C^*$-module (resp.~$H^*$-module) over the $C^*$-algebra (resp.~ $H^*$-algebra) $\A.$

We shall use the symbol $\langle W, W \rangle$ for the linear span of all inner products $\langle x,y \rangle,$ $x,y \in W.$ A Hilbert $\A$-module $W$ is said to be full if $\langle W, W \rangle$ is dense in $\A.$
Notice that $\A$ is a (full) Hilbert $\A$-module via $\langle x,y \rangle = x^*y$ for all $x,y \in \A.$

Let $w \in W$.  If $\langle w, w \rangle = e$ is a projection in $\A,$ then $we=w.$
Indeed,
$$\langle w-we, w-we \rangle = \langle w, w \rangle - \langle w, w \rangle e - e \langle w, w \rangle + e \langle w, w \rangle e = 0$$
(see the paragraph before \cite[Lemma 1]{bg2}).
This property will be used frequently later.

The main difference between Hilbert $C^*$-modules and Hilbert $H^*$-modules is the fact that Hilbert $H^*$-modules can be equipped with the structure of a complex Hilbert space. Although both structures obey the same axioms as ordinary Hilbert spaces (except that the inner product takes values in a more general structure than in the field of complex numbers), there are some properties that differ Hilbert $C^*$-modules from Hilbert spaces. For example, a closed submodule $V$ of a Hilbert $C^*$-module $W$ need not be (orthogonally) complemented, that is, $V \oplus V^{\perp} \neq W$ in general, where $V^{\perp}$ denotes $\{x \in W \, : \, \langle x,y \rangle =0 \textup{ for all } y \in V\}.$ However, Hilbert $C^*$-modules over compact operators share many nice properties with Hilbert spaces; in particular all closed submodules of such modules are complemented.

We shall deal with Hilbert $C^*$-modules over $C^*$-algebras of compact operators, and Hilbert $H^*$-modules. These structures possess orthonormal bases. More precisely, if $W$ is a Hilbert $\A$-module, where $\A$ is a $C^*$-algebra of compact operators or an $H^*$-algebra, then there exists a net $\{w_i \, : \, i \in I\}$ generating a dense submodule of $W,$ such that $\langle w_i, w_i \rangle$ is a minimal projection in $\A$, and $\langle w_i, w_j \rangle =0$ for $i \neq j.$ All these orthonormal bases have the same cardinal number which is called the orthogonal dimension of $W$ and denoted by $\dim_{\A}W.$ More details on orthonormal bases for Hilbert $C^*$-modules over $C^*$-algebras of compact operators can be found in \cite{bg2}, and for Hilbert $H^*$-modules in \cite{cmr}.

\subsection{Notation and conventions}

Let $W$ be a Hilbert $C^*$-module (resp.~ $H^*$-module) over a $C^*$-algebra (resp.~ an $H^*$-algebra) $\A$. We simply use Hilbert $\A$-module or Hilbert module over $\A$ to denote
either of them.

If $(\H, (\, . \, , . \,))$ is a Hilbert space and $\xi, \eta \in \H,$ then by $\xi \otimes \eta$  we denote the rank one operator defined by
$(\xi \otimes \eta)(\nu)=(\nu, \eta) \xi$ for all $\nu \in \H.$

All spaces are assumed to be over complex numbers.

If $G$ is an abelian group, we always use ``$+$" as its group operation.

``Orthogonally additive mapping(s)" are abbreviated as ``o.~a.~m.".

\section{$\perp$-additive mappings on Abelian Groups}\label{general}

Let $W$ and $G$ be abelian groups.
Suppose that $\perp$ is a binary relation on $W$.
We shall say that a mapping $f : W \to G$ is \textit{$\perp$-additive} if for all $x,y \in W$
$$x \perp y \implies f(x+y)=f(x)+f(y),$$
and that a mapping $F : W \times W \to G$ is \textit{$\perp$-preserving} if for all $x,y \in W$
$$x \perp y \implies F(x,y)=0.$$

Let us recall that a mapping $T: W\to G$ is called \textit{additive} if
$T(x+y)=T(x)+T(y)$ for all $x,y\in W$, a mapping $B : W \times W \to G$ is called \textit{biadditive} if it is additive in both variables,
a mapping $Q:W\to G$ is called \textit{quadratic} if
$Q(x+y)+Q(x-y)=2Q(x)+2Q(y)$ for all $x,y\in W,$
and if $W$ is a complex vector space then a mapping $S : W \times W \to G$ is called \textit{sesquilinear} if it is linear in the first
and conjugate linear in the second variable.

\begin{lem}\label{1.1}
Let $W$ be an abelian group with a binary relation $\perp$,
and $V, G$ be uniquely 2--divisible abelian groups.
Suppose that there exist additive mappings $\varphi, \psi : V \to W$ with the following properties:
\begin{align}
\label{phipsi}
\varphi(V) \perp \psi(V)\quad \text{and}\quad  (\varphi+\psi)(V) \perp (\varphi-\psi)(V).
\end{align}
Let $W_0:=\phi(V)+\psi(V) \leq W.$
If $f : W \to G$ is a $\perp$-additive mapping, then the following holds:

(i) If $f$ is odd (resp.~even), then $f$ is additive (resp.~quadratic) on $W_0.$

(ii) If $x \perp y$ implies $(-x) \perp (-y)$, then there exist mappings $T : W \to G$ and $B : W \times W \to G$ such that
$T$ is additive on $W_0$, $B$ is $\perp$-preserving symmetric biadditive on $W_0 \times W_0,$ and
$$f(x)=T(x)+B(x,x)\quad \text{for all}\quad  x\in W_0.$$
\end{lem}

\begin{proof}
(i) \,
Since $\varphi$ and $\psi$ are additive, clearly they are odd.
Now using the $\perp$-additivity of $f$, additivity of $\varphi$ and $\psi,$ as well as the properties given in (\ref{phipsi}), we obtain
\begin{eqnarray}\label{1}
&& f\big(\varphi(x)+\varphi(y)\big) + f\big(\psi(x)-\psi(y)\big) \nonumber \\
&=& f\big(\varphi(x+y)\big)+f\big(\psi(x-y)\big) \nonumber \\
&=& f\big(\varphi(x+y)+\psi(x-y)\big) \nonumber \\
&=& f\big(\varphi(x)+\varphi(y)+\psi(x)-\psi(y)\big) \nonumber \\
&=& f\big((\varphi+\psi)(x)+(\varphi-\psi)(y)\big) \nonumber \\
&=& f\big((\varphi+\psi)(x)\big)+f\big((\varphi-\psi)(y)\big) \nonumber \\
&=& f\big(\varphi(x)+\psi(x)\big)+f\big(\varphi(y)-\psi(y)\big) \nonumber \\
&=& f\big(\varphi(x)\big)+f\big(\psi(x)\big)+f\big(\varphi(y)\big)+f\big(-\psi(y)\big)
\end{eqnarray}
for all $x,y \in V$.

First assume that $f$ is odd. Switching $x$ and $y$ in (\ref{1}) gives
\begin{eqnarray}\label{2}
&& f\big(\varphi(x)+\varphi(y)\big) + f\big(\psi(y)-\psi(x)\big) \nonumber \\
&=& f\big(\varphi(y)\big)+f\big(\psi(y)\big)+f\big(\varphi(x)\big)+f\big(-\psi(x)\big).
\end{eqnarray}
Add (\ref{1}) and (\ref{2}) and use that $f$ is odd to get
$$2f\big(\varphi(x)+\varphi(y)\big) = 2f\big(\varphi(x)\big)+2f\big(\varphi(y)\big),$$
equivalently,
$$f\big(\varphi(x)+\varphi(y)\big) = f\big(\varphi(x)\big)+f\big(\varphi(y)\big)$$
as $G$ is uniquely 2--divisible.
Thus $f$ is additive on $\varphi(V).$ Then (\ref{1}) reduces to
$$f\big(\psi(x)-\psi(y)\big) = f\big(\psi(x)\big)+f\big(-\psi(y)\big),$$
that is,
$$f\big(\psi(x)+\psi(y)\big) = f\big(\psi(x)\big)+f\big(\psi(y)\big)$$
as $\psi$ is odd.
Hence, $f$ is additive on $\psi(V)$ as well. It is now easy to verify that $f$ is additive on $W_0.$

Now assume that $f$ is even. Put $y=x$ in (\ref{1}) to get
\begin{equation}\label{3}
f\big(2\varphi(x)\big)+f(0) = 2f\big(\varphi(x)\big)+2f\big(\psi(x)\big),
\end{equation}
then put $y=-x$ in (\ref{1}) to get
\begin{equation}\label{4}
f(0)+f\big(2\psi(x)\big) = 2f\big(\varphi(x)\big)+2f\big(\psi(x)\big).
\end{equation}
Comparing (\ref{3}) and (\ref{4}) yields $f\big(2\varphi(x)\big)=f\big(2\psi(x)\big),$ that is,
$f\big(\varphi(2x)\big)=f\big(\psi(2x)\big)$ for all $x \in V.$
Since $V$ is uniquely 2--divisible, we have
\begin{equation}
\label{fphipsi}
f\big(\varphi(x)\big)=f\big(\psi(x)\big)\quad \text{for all}\quad  x \in V.
\end{equation}
Then $f\big(\psi(x)-\psi(y)\big)=f\big(\varphi(x)-\varphi(y)\big)$ for all $x,y \in V.$
This together with (\ref{1}) implies
\begin{eqnarray}
f\big(\varphi(x)+\varphi(y)\big)+f\big(\varphi(x)-\varphi(y)\big)
=2f\big(\varphi(x)\big)+2f\big(\varphi(y)\big), \nonumber
\end{eqnarray}
so $f$ is quadratic on $\varphi(V).$ Then it follows from \eqref{fphipsi} that $f$ is also quadratic on $\psi(V)$.
Therefore $f$ is quadratic on $W_0$.

\medskip
(ii) \,
Set
\[
T(x)=\frac{1}{2}\big(f(x)-f(-x)\big),\ F(x)=\frac{1}{2}\big(f(x)+f(-x)\big) \quad\text{for all}\quad {x\in W}.
\]
Then $T$ is odd and $\perp$-additive.
By (i), $T$ is additive on $W_0.$
Furthermore, $F$ is even and $\perp$-additive.
Again by (i), $F$ is quadratic on $W_0.$
Then $F(0)=0,$ so
$$F(x+x)+F(x-x)=2F(x)+2F(x)$$
yields $F(2x)=4F(x)$ for all $x \in W_0.$

Define
\[
B(x,y)=\frac{1}{4}\big(F(x+y)-F(x-y)\big)\quad\text{for all}\quad {x,y\in W}.
\]
 Since $F$ is even, $B$ is symmetric.
 It is well-known that $B$ is biadditive (on $W_0$),
 but in the sequel we prove this fact for reader's convenience.
Obviously, $B(x,0)=B(0,x)=0$ and $B(x,x)=\frac{1}{4}\big(F(2x)-F(0)\big)=F(x)$ for all $x \in W_0.$
Since $F$ is quadratic on $W_0$, for all $x,y,u \in W_0$ we have
\begin{eqnarray}
&&4B(x+y,2u)=F(x+y+2u)-F(x+y-2u) \nonumber \\
&=&F\big((x+u)+(y+u)\big)+F\big((x+u)-(y+u)\big) \nonumber \\
&&-F\big((x-u)-(y-u)\big)-F\big((x-u)+(y-u)\big) \nonumber \\
&=&2\big(F(x+u)+F(y+u)\big)-2\big(F(x-u)+F(y-u)\big) \nonumber \\
&=&8B(x,u)+8B(y,u). \nonumber
\end{eqnarray}
Since $G$ is uniquely 2--divisible, this implies
\begin{equation}\label{5}
B(x+y,2u)=2B(x,u)+2B(y,u).
\end{equation}
Inserting $y=0$ and $x=z$ yields
$$B(z,2u)=2B(z,u).$$
If we put $x+y$ instead of $z$, using (\ref{5}) we get
$$B(x+y,u)=B(x,u)+B(y,u).$$
Hence, $B$ is biadditive on $W_0 \times W_0.$ Finally,
$$f(x)=T(x)+F(x)=T(x)+B(x,x)$$
for all $x \in W_0.$ Notice that, for all $x,y \in W_0,$ $x \perp y$ implies
\begin{eqnarray}
&&2B(x,y)=B(x,y)+B(y,x) \nonumber \\
&=&B(x+y,x+y)-B(x,x)-B(y,y) \nonumber \\
&=&\big(f(x+y)-T(x+y)\big)-\big(f(x)-T(x)\big)-\big(f(y)-T(y)\big) \nonumber \\
&=&\big(f(x+y)-f(x)-f(y)\big)-\big(T(x+y)-T(x)-T(y)\big)=0. \nonumber
\end{eqnarray}
Hence $B(x,y)=0$, namely, $B$ is $\perp$-perserving on $W_0\times W_0$.
\end{proof}

\begin{rem}
\label{TB}
By the definitions of $T$ and $B$, it is easy to see that they are uniquely determined by $f$. Actually,
\begin{align*}
T(x)&=\frac{1}{2}\big(f(x)-f(-x)\big),\\
B(x,y)&=\frac{1}{8}\big(f(x+y)+f(-x-y)-f(x-y)-f(-x+y)\big)
\end{align*}
for all $x,y \in W.$
The reason why we only have
\[
f(x)=T(x)+B(x,x)
\]
for all $x\in W_0$, instead of $W$, is because it is only known that $F(x)=B(x,x)$ for all $x\in W_0$.
\end{rem}

\begin{lem}\label{1.3}
Let $W, V, G$ be  normed spaces, and $\perp$ be a binary relation on $W$ such that $x \perp y$ implies $(-x) \perp (-y)$.
Suppose that there are continuous linear mappings $\varphi, \psi : V \to W$ with the following properties:
\begin{equation}\label{ab}
\varphi(V) \perp \psi(V) \textup{ and } (\varphi + \lambda \psi)(V) \perp (\varphi - \lambda \psi)(V) \textup{ for } \lambda \in \{1,i\}.
\end{equation}
Let $W_0 := \varphi(V)+\psi(V) \leq W.$

If $f : W \to G$ is a continuous $\perp$-additive mapping,
then there exist continuous mappings $T : W \to G$ and $S : W \times W \to G$
such that $T$ is additive on $W_0$, $S$ is sesquilinear on $W_0 \times W_0$
with the property that for all $x,y \in W_0$,
$x \perp y$ implies $S(x,y)+S(y,x)=0,$
and
$$f(x)=T(x)+S(x,x)\quad \text{for all}\quad x\in W_0.$$
Furthermore, if we also assume that $x \perp y$ implies $x \perp iy$, then $S$ is $\perp$-preserving on $W_0 \times W_0.$
\end{lem}

\begin{proof}
By Lemma \ref{1.1} (ii),
there exist mappings $T : W \to G$ and $B : W \times W \to G$ such that $T$ is additive on $W_0,$
$B$ is $\perp$-preserving symmetric biadditive on $W_0 \times W_0,$ and
$$f(x)=T(x)+B(x,x)\quad \text{for all} \quad x\in W_0.$$
Since $f$ is continuous, clearly so are $T$ and $B$ (see Remark \ref{TB}).

Since $\psi$ is linear and $B$ is $\perp$-preserving on $W_0 \times W_0,$ it follows from (\ref{ab}) that for all $x,y \in V$ and $\lambda \in \{1,i\}$ we have
\begin{eqnarray}
0 &=& B\big((\varphi+\lambda\psi)(x), (\varphi-\lambda\psi)(y)\big) \nonumber \\
&=& B\big(\varphi(x)+\lambda\psi(x), \varphi(y)-\lambda\psi(y)\big) \nonumber \\
&=& B\big(\varphi(x), \varphi(y)\big)+B\big(\psi(\lambda x), \varphi(y)\big) \nonumber \\
&&- B\big(\varphi(x), \psi(\lambda y)\big) - B\big(\lambda \psi(x), \lambda \psi(y)\big) \nonumber \\
&=& B\big(\varphi(x), \varphi(y)\big) - B\big(\lambda \psi(x), \lambda \psi(y)\big). \nonumber
\end{eqnarray}
This implies
$$B\big(i \psi(x), i \psi(y)\big) = B\big(\varphi(x), \varphi(y)\big) = B\big(\psi(x), \psi(y)\big)$$
and
\begin{align*}
B\big(i\varphi(x), i\varphi(y)\big)
&=B\big(\varphi(ix),\varphi(iy)\big)=B\big(\psi(ix), \psi(iy)\big) \nonumber \\
&=B\big(i\psi(x), i\psi(y)\big) = B\big(\varphi(x), \varphi(y)\big). \nonumber
\end{align*}
Hence,
$$B(ix,iy)=B(x,y)$$
for all $x,y \in \varphi(V) + \psi(V) = W_0.$

Since $B$ is biadditive and continuous on $W_0 \times W_0,$ it is also $\mathbb{R}$-bilinear on $W_0 \times W_0.$
Define $S : W \times W \to G$ by
\[
S(x,y)=B(x,y)+iB(x,iy).
\]
Then, for all $x,y \in W_0,$
\begin{align*}
S(ix,y) &= B(ix,y)+iB(ix,iy)=B(ix,y)+iB(x,y) \nonumber \\
&=i \big(B(x,y)-iB(ix,y)\big) = i \big(B(x,y)+iB(x,iy)\big)\\
&=iS(x,y), \nonumber
\end{align*}
and analogously
$$S(x,iy)=-iS(x,y).$$
Since the mapping $B$ is continuous $\mathbb{R}$-bilinear on $W_0 \times W_0,$ the mapping $S$ is continuous $\mathbb{R}$-bilinear on $W_0 \times W_0$ as well.
However, from the above we conclude that $S$ is continuous sesquilinear on $W_0 \times W_0.$

Also, notice that
\begin{align*}
S(x,y)+S(y,x)
&=B(x,y)+iB(x,iy)+B(y,x)+iB(y,ix) \nonumber \\
&=2B(x,y)+iB(x,iy)+iB(ix,y)=2B(x,y) \nonumber
\end{align*}
for all $x,y \in W_0.$
In particular,  we get
$$S(x,x)=B(x,x)\quad \text{for all} \quad x\in W_0.$$
Then $S$ is a continuous sesquilinear mapping on $W_0 \times W_0$ with the properties that for all $x,y \in W_0,$
$$x \perp y \implies S(x,y)+S(y,x)=0$$
and
$$f(x)=T(x)+S(x,x) \quad \text{for all} \quad x\in W_0.$$

Furthermore, if $x \perp y$ implies $x \perp iy,$ then for all $x,y \in W_0$ satisfying $x \perp y$ we have
$$S(x,y)+S(y,x)=0 \textup{ and } -iS(x,y)+iS(y,x)=0.$$
Hence $S(x,y)=0,$ that is, $S$ is $\perp$-preserving on $W_0 \times W_0.$
\end{proof}

\medskip

By the definition of $S$ and Remark \ref{TB}, we conclude that $S$ is also uniquely determined by $f$ and
\begin{align*}
S(x,y)
&=\frac{1}{8}\big(f(x+y)+if(x+iy)-f(x-y)-if(x-iy) \nonumber \\
&\quad + f(-x-y)+if(-x-iy)-f(-x+y)-if(-x+iy)\big) \nonumber
\end{align*}
for all $x,y \in W.$

It should be also mentioned that the mapping $T$ is $\mathbb{R}$-linear on $W_0$ since it is continuous and additive on $W_0,$
but it is not $\mathbb{C}$-linear in general.

\section{O.~A.~M.~on Hilbert modules}\label{modules}

Let $(W, \langle \, . \, , \, . \, \rangle)$ be a Hilbert $\A$-module and let $G$ be an abelian group.
We shall study $\perp$-additive mappings and $\perp$-preserving mappings for the binary relation $\perp$ on $W$ given by
$$x \perp y \Longleftrightarrow \langle x,y \rangle =0.$$
A mapping $f : W \to G$ is said to be \textit{orthogonally additive} if
$$\langle x,y \rangle = 0 \implies f(x+y)=f(x)+f(y).$$
A mapping $B : W \times W \to G$ is said to be \textit{orthogonality preserving} if
$$\langle x,y \rangle =0 \implies B(x,y)=0.$$

A \textit{morphism} between Hilbert $\A$-modules $V$ and $W$ is a mapping
$\varphi : V \to W$ satisfying $\langle \varphi(x), \varphi(y) \rangle = \langle x,y \rangle$
for all $x,y \in V.$
It is clear that morphisms are continuous mappings, and it is not difficult to verify that they are also $\A$-linear mappings, that is, linear mappings satisfying $\varphi(xa)=\varphi(x)a$ for all $x \in V$ and $a \in \A.$

\begin{thm}\label{2.1}
Let $W$ be a Hilbert $\A$-module, $V$ be a submodule of $W$,
and $\varphi : V \to W$ be a morphism such that $\varphi(V) \subseteq V^{\perp}.$
Let $W_0 := V \oplus \varphi(V) \leq W.$
Suppose that $G$ is a uniquely 2--divisible abelian group and that $f : W \to G$ is an o.~a.~m.
Then the following holds:

(i) There exist mappings $T : W \to G$
and $B : W \times W \to G$ such that $T$ is additive on $W_0,$ $B$ is symmetric biadditive orthogonality preserving on $W_0 \times W_0,$ and
$$f(x)=T(x)+B(x,x)\quad \text{for all}\quad x \in W_0.$$

(ii) If $G$ is a normed space and $f$ is continuous, then there exist continuous mappings
$T : W \to G$ and  $S : W \times W \to G$ such that $T$ is additive on $W_0,$ $S$ is sesquilinear
orthogonality preserving on $W_0 \times W_0$,  and
$$f(x)=T(x)+S(x,x)\quad \text{for all}\quad x \in W_0.$$
\end{thm}

\begin{proof}
Set $x \perp y$ if and only if $\langle x,y \rangle =0.$ Let $id : V \to V$ be the identity mappping.
Since $\varphi(V) \subseteq V^{\perp},$ we have $\langle \varphi(V), id(V) \rangle =0.$
Furthermore, for all $x,y \in V,$ and $\lambda \in \{1,i\},$
$$\langle (\varphi+\lambda \cdot id)(x), (\varphi- \lambda \cdot id)(y) \rangle = \langle \varphi(x), \varphi(y) \rangle - \langle x, y \rangle =0.$$
Then (i) and (ii)  follow from Lemma \ref{1.1} (ii) and Lemma \ref{1.3}, respectively.
\end{proof}


We write $U \sim V$ if $U$ and $V$ are unitarily equivalent Hilbert $C^*$-modules over a $C^*$-algebra $\A,$ that is,
if there exists a mapping $u : U \to V$ such that there is a mapping $u^* : V \to U$ satisfying $\langle ux, y \rangle = \langle x, u^*y \rangle$
for all $x \in U,$ $y \in V,$ and
$$u^*u=id_U, \qquad uu^*=id_V.$$
It is clear that $u$ is surjective and $\langle u(x), u(y) \rangle = \langle x, y \rangle$ for all $x,y \in U.$

A closed submodule $V$ of a Hilbert $C^*$-module $W$ is said to be \textit{complemented} if $W=V\oplus V^{\perp},$
and \textit{fully complemented} if $V$ is complemented and $V^{\perp} \sim W.$

\begin{cor}\label{2.2}
Let $V$ be a fully complemented submodule of a Hilbert $C^*$-module $W,$
$G$ be a uniquely 2--divisible abelian group, and $f : W \to G$ be an o.~a.~m.
Then there exist mappings $T : W \to G$ and $B : W \times W \to G$ such that $T$ is additive on $V,$  $B$ is  symmetric biadditive orthogonality preserving
on $V \times V,$ and
$$f(x)=T(x)+B(x,x)\quad \text{for all}\quad x \in V.$$

Furthermore, if $G$ is a normed space and $f$ is continuous then there exist a continuous mapping $T : W \to G$ which is
additive on $V$, and a continuous mapping $S : W \times W \to G,$ which is sesquilinear and orthogonality preserving on $V \times V,$ such that
$$f(x)=T(x)+S(x,x)\quad \text{for all}\quad x \in V.$$
\end{cor}

\begin{proof}
Since $V$ is a fully complemented submodule of $W$ there exists a linear operator $u : W \to V^{\perp}$ such that $\langle u(x), u(y) \rangle = \langle x, y \rangle$ for all $x,y \in W.$
Set $\varphi = u \vert_V.$
Then $\varphi : V \to V^{\perp}$ satisfies $\langle \varphi(x), \varphi(y) \rangle = \langle x, y \rangle$ for all $x,y \in V.$
So it remains to apply Theorem \ref{2.1}.
\end{proof}






\section{O.~A.~M.~on Hilbert $\K(\H)$-modules and Hilbert $\H\S(\H)$-modules}\label{K(H)&HS(H)}

In this section, \textbf{$\A$ always denotes $\K(\H)$ or $\H\S(\H)$}.
Let $W$ be a Hilbert $\A$-module and $e \in \A$ be a rank one projection in $\A$. Then there exists an orthonormal basis $\{w_i \, : \, i \in I\}$ for $W$ such that $\langle w_i, w_i \rangle = e$ for all $i \in I$ (see \cite[Remark 4 (d)]{bg2} for Hilbert $\K(\H)$-modules, and \cite[Proposition 1.5]{bg1} for Hilbert $\H\S(\H)$-modules).
The following lemma will allow us to deal with yet another suitable orthonormal basis for $W.$

\begin{lem}\label{bases}
Let
$W$ be a Hilbert $\A$-module with $\dim_{\A} W \leq \dim \H,$ and
$\{\xi_i \, : \, i \in I\}$ be an orthonormal basis for $\H.$
Then there exists an orthonormal basis $\{w_i \, : \, i \in J \subseteq I\}$ for $W$ such that $\langle w_i, w_i \rangle = \xi_i \otimes \xi_i$ for all $i \in J.$
\end{lem}

\begin{proof}
Let us fix an arbitrary $j_0 \in I.$ Then
there exists an orthonormal basis $\{g_i \, : \, i \in J\}$ for $W$ such that $\langle g_i, g_i \rangle = \xi_{j_0} \otimes \xi_{j_0}$ for all $i \in J.$
Since $\dim_{\A} W \leq \dim \H$, we assume $J \subseteq I.$
Define, for all $i \in J,$
$$w_i = g_i(\xi_{j_0} \otimes \xi_i).$$
Then $\langle w_i, w_j \rangle = 0$ if $i \neq j$ and
\begin{eqnarray}
\langle w_i, w_i \rangle &=& \langle g_i(\xi_{j_0} \otimes \xi_i), g_i(\xi_{j_0} \otimes \xi_i) \rangle \nonumber \\
&=& (\xi_i \otimes \xi_{j_0})\langle g_i, g_i \rangle (\xi_{j_0} \otimes \xi_i) \nonumber \\
&=& (\xi_i \otimes \xi_{j_0})(\xi_{j_0} \otimes \xi_{j_0})(\xi_{j_0} \otimes \xi_i) \nonumber\\
&=& \xi_i \otimes \xi_i \nonumber
\end{eqnarray}
for all $i \in J.$
Furthermore, for all $x \in W,$
\begin{eqnarray}
x &=& \sum_{i \in J} g_i \langle g_i, x \rangle \nonumber \\
&=& \sum_{i \in J} g_i \langle g_i(\xi_{j_0} \otimes \xi_{j_0}), x \rangle \nonumber \\
&=& \sum_{i \in J} g_i(\xi_{j_0} \otimes \xi_{j_0})\langle g_i, x \rangle \nonumber \\
&=& \sum_{i \in J} g_i(\xi_{j_0} \otimes \xi_i)(\xi_i \otimes \xi_{j_0})\langle g_i, x \rangle \nonumber \\
&=& \sum_{i \in J} g_i(\xi_{j_0} \otimes \xi_i)\langle g_i(\xi_{j_0} \otimes \xi_i), x \rangle \nonumber\\
&=& \sum_{i \in J} w_i \langle w_i, x \rangle. \nonumber
\end{eqnarray}
By \cite[Theorem 1]{bg2}, $\{w_i \, : \, i \in J\}$ is an orthonormal basis for $W.$
\end{proof}

\begin{rem}\label{rem3.2}
Let $\H$ be a Hilbert space with $\dim\H=\aleph_0$,
$W$ be a Hilbert $\A$-module such that $\dim_{\A} W = \aleph_0$, and
$\{\xi_i \, : \, i \in \bN\}$ be an orthonormal basis for $\H.$
By Lemma \ref{bases} there exists an orthonormal basis $\{w_i \, : \, i \in \bN\}$ for $W$ such that $\langle w_i, w_i \rangle = \xi_i \otimes \xi_i$ for all $i \in \bN.$ Let $a \in \A.$
Since
\begin{align*}
&\Vert \sum_{i=m}^n w_i a \Vert_W^2
= \Vert \langle \sum_{i=m}^n w_i a, \sum_{i=m}^n w_i a \rangle \Vert \nonumber \\
=& \Vert \sum_{i=m}^n a^* \langle w_i, w_i \rangle a \Vert
= \Vert \sum_{i=m}^n a^* (\xi_i \otimes \xi_i) a \Vert, \nonumber
\end{align*}
and $\sum_{i=1}^{\infty} a^* (\xi_i \otimes \xi_i) a = a^*a,$ the sequence $(\sum_{i=1}^n w_ia)_{n=1}^\infty$ is a Cauchy sequence in $W,$ so it converges.
Hence, for all $a \in \A$, $\sum_{i=1}^{\infty} w_ia \in W$ and $\langle \sum_{i=1}^{\infty} w_ia, \sum_{i=1}^{\infty} w_ia \rangle = a^*a.$
\end{rem}

Before giving the main result of this section, we provide a representation result of sesquilinear orthogonality preserving mappings $S:W\times W\to G$.
This result is of independent interest.
\begin{prop}
\label{P:repS}
Let $W$ be a Hilbert $\A$-module and $G$ be a normed space. If $S:W\times W\to G$ is a continuous sesquilinear orthogonality preserving
mapping, then there is a unique linear mapping $\Phi:\langle W, W\rangle \to G$ such that
\[
S(x,y)=\Phi(\langle y, x\rangle) \quad \text{for all}\quad x,y\in W.
\]

Furthermore, if $\H$ is finite dimensional or $\dim{\H}=\dim_{\A}W=\aleph_0,$ then $\Phi$ can be extended to a continuous linear mapping on $\A$.
\end{prop}

\begin{proof}
Let $\{\xi_i \, : \, i \in I\}$ be an orthonormal basis for $\mathcal{H}$ and let $e_i = \xi_i \otimes \xi_i$ for all $i \in I.$
Fix an arbitrary $i_0 \in I.$
Let $\{w_j \, : \, j \in J\}$ be an orthonormal basis for $W$ such that $\langle w_j, w_j \rangle = e_{i_0}$ for all $j \in J.$
Then for all $j,k \in J,$ $j \neq k,$ and all $a,b \in \A$ we have
\begin{align*}
 \langle w_ja - w_ka, w_jb + w_kb \rangle &= a^*\langle w_j, w_j \rangle b - a^*\langle w_k, w_k \rangle b \nonumber \\
&=a^*e_{i_0}b-a^*e_{i_0}b=0, \nonumber
\end{align*}
hence
$$0=S(w_ja - w_ka, w_jb + w_kb)=S(w_ja, w_jb)-S(w_ka, w_kb).$$
In particular, for $b=e_{i_0}$, one obtains
$$S(w_ja, w_j)=S(w_ja, w_je_{i_0})=S(w_ka, w_ke_{i_0})=S(w_ka, w_k)$$
for all $j,k \in J$.
Hence the mapping $\Phi_{i_0} : \A \to G,$ defined by
$$\Phi_{i_0}(a)=S(w_ka, w_k),$$
does not depend on $k \in J.$
It is clear that $\Phi_{i_0}$ is linear.
Notice that
\begin{align}\nonumber
\Vert \Phi_{i_0}(a)\Vert
&= \Vert S(w_ka, w_k) \Vert \leq \Vert S \Vert \cdot \Vert w_k \Vert_{W}^2 \cdot \Vert a \Vert \\ \label{bound}
&= \Vert S \Vert \cdot \Vert e_{i_0} \Vert \cdot \Vert a \Vert.
\end{align}
Let us remark that if $\A=\K(\H)$ then each (minimal) projection has norm one, which is not true in general if $\A=\H\S(\H),$
but in both cases $\Phi_{i_0}$ is bounded and $\Vert \Phi_{i_0} \Vert \leq \Vert S \Vert \cdot \Vert e_{i_0} \Vert.$
Furthermore,
\begin{eqnarray}
&&\Phi_{i_0}(\langle ye_{i_0},x \rangle)=S(w_k\langle ye_{i_0},x\rangle, w_k)=\sum_{j \in J} S(w_k\langle ye_{i_0},w_j \rangle\langle w_j,x \rangle, w_k) \nonumber \\
&=&\sum_{j \in J} S(w_ke_{i_0}\langle y,w_j \rangle e_{i_0} \langle w_j,x \rangle, w_k)=\sum_{j \in J} S(w_k(\lambda_j e_{i_0}) \langle w_j,x \rangle, w_k) \nonumber \\
&=&\sum_{j \in J} \lambda_j S(w_k\langle w_j,x \rangle, w_k)=\sum_{j \in J} S(w_k\langle w_j,x \rangle, w_k(\overline{\lambda_j}e_{i_0})) \nonumber \\
&=&\sum_{j \in J} S(w_k\langle w_j,x \rangle, w_k(e_{i_0}\langle w_j, y\rangle e_{i_0})) = \sum_{j \in J} S(w_k\langle w_j,x \rangle, w_k\langle w_j, ye_{i_0}\rangle) \nonumber \\
&=&\sum_{j \in J} S(w_j\langle w_j,x \rangle, w_j\langle w_j, ye_{i_0}\rangle)= S(\sum_{j \in J} w_j\langle w_j, x \rangle, \sum_{j \in J} w_j\langle w_j, ye_{i_0}\rangle) \nonumber \\
&=&S(x, ye_{i_0}). \nonumber
\end{eqnarray}
Since $i_0 \in I$ is arbitrary, it follows that
\begin{equation}\label{gl}
S(x,y)=S(x, \sum_{i \in I}ye_i)=\sum_{i \in I}S(x, ye_i)=\sum_{i \in I}\Phi_i(\langle ye_i,x \rangle).
\end{equation}

Let us define $\Phi : \langle W, W \rangle \to G$ by
$$\Phi(a)=\sum_{i \in I}\Phi_i(e_ia).$$
By (\ref{gl}), the mapping $\Phi$ is well-defined and
$$S(x,y)=\Phi(\langle y, x \rangle)$$
for all $x,y \in W.$
Since all $\Phi_i$ are linear, $\Phi$ is linear as well. Uniqueness of such $\Phi$ is obvious.

If $\H$ is finite dimensional then $\sum_{i \in I}\Phi_i(e_ia)$ converges for all $a \in \A.$
If $\dim{\H}=\dim_{\A}W=\aleph_0$ then by Lemma \ref{bases} there exists an orthonormal basis $\{v_i \, : \, i \in I\}$ for $W$ such that $\langle v_i, v_i \rangle = \xi_i \otimes \xi_i$ for all $i \in I.$
By Remark \ref{rem3.2},
$\sum_{j \in I}v_ja \in W$ for all $a \in \A,$ so for all $a,b \in \A$ we have
\begin{eqnarray}
&&S(\sum_{j \in I}v_jb, \sum_{j \in I}v_ja^*) = \sum_{i \in I}\Phi_i(\langle \sum_{j \in I}v_ja^*e_i, \sum_{j \in I}v_jb) \nonumber \\
&=& \sum_{i \in I} \Phi_i (\sum_{j \in I}e_ia\langle v_j, v_j \rangle b) = \sum_{i \in I} \Phi_i (\sum_{j \in I}e_ia(\xi_j \otimes \xi_j)b) \nonumber \\
&=& \sum_{i \in I} \Phi_i(e_iab). \nonumber
\end{eqnarray}
Hence $\sum_{i \in I}\Phi_i(e_ia)$ converges for all $a \in \A^2=\A.$
It means that in the cases when $\dim{\H}$ is finite or $\dim{\H}=\dim_{\A}W=\aleph_0$ we can extend $\Phi$ from $\langle W, W \rangle$ to $\A$ if we define
$$\Phi(a)=\sum_{i \in I}\Phi_i(e_ia).$$

Finally, let us prove that $\Phi : \A \to G$ is bounded. If $\H$ is finite dimensional, this immediately follows from (\ref{bound}).
Now assume that $\dim{\H}=\dim_{\A}W=\aleph_0.$ Then for every $a \in \A$ we have, by Remark \ref{rem3.2},
\begin{eqnarray}
&& \Vert \Phi(a^*a) \Vert = \Vert \Phi( \langle \sum_{i \in I} v_ia, \sum_{i \in I} v_ia \rangle) \Vert \nonumber \\
&=& \Vert S(\sum_{i \in I} v_ia, \sum_{i \in I} v_ia) \Vert \leq \Vert S \Vert  \Vert \sum_{i \in I} v_ia \Vert_{W}^2 \nonumber \\
&=& \Vert S \Vert \Vert \langle \sum_{i \in I} v_ia, \sum_{i \in I} v_ia \rangle \Vert = \Vert S \Vert \Vert a^*a \Vert. \nonumber
\end{eqnarray}
Thus $\Phi$ is bounded on positive elements on $\A.$ Therefore it is bounded on $\A.$
\end{proof}

We are now ready to prove our main result of this section.

\begin{thm}\label{3.1}
Let $W$ be a Hilbert $\A$-module such that $\dim_{\A}W\ge 2.$
Let $G$ be a uniquely 2--divisible abelian group, and $f : W \to G$ be an o.~a.~m.
Then the following holds:

(i) There exist a unique additive mapping $T : W \to G$ and a unique symmetric biadditive orthogonality preserving
mapping $B : W \times W \to G$ such that
$$f(x)=T(x)+B(x,x)\quad \text{for all}\quad x \in W.$$

(ii) If $G$ is a normed space and $f$ is continuous, then there are a unique continuous additive mapping $T : W \to G$
and a unique linear mapping $\Phi : \langle W, W \rangle \to G$ such that
$$f(x)=T(x)+\Phi(\langle x,x \rangle)\quad \text{for all}\quad x \in W.$$
Furthermore, if $\H$ is finite dimensional or $\dim{\H}=\dim_{\A}W=\aleph_0,$ then $\Phi$ can be extended to a continuous linear mapping
on $\A$.
\end{thm}

\begin{proof}
(i) \,
First assume that $W$ is either finite dimensional with $\dim_{\A}W = 2n,$ or $\dim_{\A}W \ge \aleph_0.$
If $\dim_{\A}W = 2n$ then let $V$ be a closed submodule of $W$ such that $\dim_{\A}W = n$; if $\dim_{\A}W \ge \aleph_0$ then let $V$ be a closed submodule of $W$ such that $\dim_{\A}V = \dim_{\A}V^{\perp} = \dim_\A W.$
Let $\{w_i \, : \, i \in I\}$ be an orthonormal basis for $W$ such that $\langle w_i, w_i \rangle =e$ for all $i \in I,$
where $e$ is a fixed rank one projection in $\A$.
Let $\{w_i \, : \, i \in I_1  \subseteq I\}$ be an orthonormal basis for $V$ and $\{w_i \, : \, i \in I_2 \subseteq I\}$ be an orthonormal basis for $V^{\perp}.$
Let $\varphi : V \to V^{\perp}$ be an isomorphism between the bases of $V$ and $V^{\perp}.$ It remains to apply Theorem \ref{2.1} (i).
Notice that $V \oplus \varphi(V) = V \oplus V^{\perp} =W.$

Now assume that $W$ is finite dimensional with $\dim_{\A}W = 2n+1.$
From the above we conclude that the desired conclusion is true for $f$ restricted to any $2n$-dimensional closed submodule of $W.$
Let $X$ be a closed submodule of $W$ such that $\dim_{\A}X= 1.$ Then $\dim_{\A}X^{\perp} = 2n.$
Let $Z$ be a closed submodule of $W$ such that $\dim_{\A}Z = 2,$ and $X \subset Z.$
The statement is true both on $Z$ and $X^{\perp},$ hence on $W$.

(ii) \,
Combining the proofs of (i) above and Theorem \ref{2.1} (ii), we can find a unique continuous additive mapping
$T : W \to G$ and a unique continuous sesquilinear orthogonality preserving mapping $S : W \times W \to G$ such that
$$f(x)=T(x)+S(x,x)\quad \text{for all}\quad x \in W.$$
Then apply Proposition \ref{P:repS} to end the proof.
\end{proof}

We should mention that the condition $\dim_{\A} W \geq 2$ is essential in Theorem \ref{3.1}, as shown
in the following example.

\begin{eg}
Let $\H$ be an infinite dimensional Hilbert space.
Then $\H$ is a Hilbert $\A$-module with respect to the $\A$-valued inner product given by $\langle \xi,\eta\rangle=\eta\otimes\xi.$ It is known that $\dim_{\A}\H=1$ (\cite[Example 1]{bg2} and \cite[Example 2.3]{bg1}). Notice that $\langle \xi, \eta\rangle=0$ if and only if $\xi=0$ or $\eta=0$. Then every odd mapping on $\H$ (taking values in a uniquely 2--divisible abelian group) is orthogonally additive, but not additive in general. For example, fix $0\neq \eta_0\in \H$ and define $f(\xi)=(\xi,\eta_0)\xi\otimes\xi$ for all $\xi\in \H$.
\end{eg}

The following example shows that in the case when $\H$ is infinite dimensional and $\dim_{\A}W$ is finite, the mapping $\Phi$ from Theorem \ref{3.1} cannot be extended to a continuous linear mapping on $\A.$

\begin{eg}
Let $\H$ be an infinite dimensional separable Hilbert space. Let $W = \H \oplus \H$ be a Hilbert $\A$-module with the coordinate operations and the $\A$-valued inner product given by
$$\langle (\xi_1, \xi_2), (\eta_1, \eta_2) \rangle = \eta_1 \otimes \xi_1 + \eta_2 \otimes \xi_2.$$
Then $\dim_{\A}W=2$ (see \cite[Theorem 3]{bg2} and \cite[Section 2]{cmr}).
To distinguish the above notation, we use $(\, . \, , \, . \,)_\H$ for the inner product on $\H$.
Define $f : W \to \bC$ by
$$f\big((\xi_1,\xi_2)\big) = (\xi_1, \xi_1)_\H + (\xi_2, \xi_2)_\H\quad \text{for all}\quad \xi_1, \xi_2 \in \H.$$

We claim that $f$ is an orthogonally additive mapping. Indeed,
let $\xi_1, \xi_2, \eta_1, \eta_2 \in \H$ be such that
$$0=\langle (\xi_1, \xi_2), (\eta_1, \eta_2) \rangle = \eta_1 \otimes \xi_1 + \eta_2 \otimes \xi_2.$$
Then one can easily check that
$$(\xi_1, \eta_1)_\H + (\eta_1, \xi_1)_\H + (\xi_2, \eta_2)_\H + (\eta_2, \xi_2)_\H=0.$$
Thus
\begin{eqnarray}
&& f\big((\xi_1, \xi_2)+(\eta_1, \eta_2)\big) = f\big((\xi_1 + \eta_1, \xi_2 + \eta_2)\big) \nonumber \\
&=& ( \xi_1+\eta_1, \xi_1+\eta_1 )_\H + ( \xi_2+\eta_2, \xi_2+\eta_2 )_\H \nonumber \\
&=& ( \xi_1, \xi_1 )_\H + ( \xi_1, \eta_1 )_\H + ( \eta_1, \xi_1 )_\H + ( \eta_1, \eta_1 )_\H \nonumber \\
&& + ( \xi_2, \xi_2 )_\H + ( \xi_2, \eta_2 )_\H + ( \eta_2, \xi_2 )_\H + ( \eta_2, \eta_2 )_\H \nonumber \\
&=& ( \xi_1, \xi_1 )_\H + ( \xi_2, \xi_2 )_\H + ( \eta_1, \eta_1 )_\H + ( \eta_2, \eta_2 )_\H \nonumber \\
&=& f\big((\xi_1,\xi_2)\big) + f\big((\eta_1,\eta_2)\big). \nonumber
\end{eqnarray}
This shows that $f$ is an orthogonally additive mapping. It is clear that $f$ is even.

In what follows, we prove that there is no continuous linear mapping $\Phi:\A\to\mathbb{C}$
such that $f(w)=\Phi(\langle w,w \rangle)$ for all $w \in W.$
To the contrary, assume that there is such a mapping $\Phi$.
Let $\{ \xi_n \, : \, n \in \mathbb{N}\}$ be an orthonormal basis for $\H$ and let $E_n = \xi_n \otimes \xi_n$ for all $n \in \mathbb{N}.$
For each $n \in \mathbb{N}$, we set $T_n = \sum_{k=1}^n \frac{1}{k} E_k \in \langle W, W \rangle.$
Since the sequence $(T_n)$ converges to $T = \sum_{k=1}^{\infty} \frac{1}{k} E_k \in \A$ and $\Phi$ is continuous, we conclude that the sequence $(\Phi(T_n))$ converges as well. However,
\begin{align*}
\Phi(T_n)&= \sum_{k=1}^n \frac{1}{k} \Phi(E_k) = \frac{1}{2} \sum_{k=1}^n \frac{1}{k} \Phi(2E_k) \nonumber \\
&= \frac{1}{2} \sum_{k=1}^n \frac{1}{k} \Phi(\xi_k \otimes \xi_k + \xi_k \otimes \xi_k) \\
&= \frac{1}{2} \sum_{k=1}^n \frac{1}{k} \Phi(\langle (\xi_k, \xi_k), (\xi_k, \xi_k) \rangle) \nonumber \\
&= \frac{1}{2} \sum_{k=1}^n \frac{1}{k} f\big((\xi_k, \xi_k)\big) \\
&= \sum_{k=1}^n \frac{1}{k} \nonumber
\end{align*}
does not converge; a contradiction.
\end{eg}

The following result is an immediate consequence of Theorem \ref{3.1} (see \cite[Example 2]{bg2}).

\begin{cor}\label{3.4}
Let $\H$ be a Hilbert space with $2 \leq \dim \H \leq \aleph_0,$
and the orthogonality on $\A$ be defined by
$$x \perp y \Longleftrightarrow x^*y=0.$$
Assume that $G$ is a uniquely 2--divisible abelian group, and that $f : \A \to G$ is an o.~a.~m. Then the following holds:
\begin{itemize}
\item[(i)]
There exist a unique additive mapping $T : \A \to G$ and a unique symmetric biadditive orthogonality preserving mapping $B : \A \times \A \to G$ such that
$$f(x)=T(x)+B(x,x)\quad\text{for all}\quad x \in \A.$$
\item[(ii)]
If $G$ is a normed space and $f$ is continuous, then $T$ is continuous and there exists a unique continuous linear mapping
$\Phi : \A \to G$ such that
$$f(x)=T(x)+\Phi(x^*x)\quad\text{for all}\quad x \in \A.$$
\end{itemize}
\end{cor}

The following example demonstrates that the underlying algebra is $\K(\H)$ (instead of just being a $C^*$-algebra of compact operators) is essential in  Corollary \ref{3.4} (and Theorem \ref{3.1}).
The same example shows that we cannot take an arbitrary $H^*$-algebra instead of $\H\S(\H).$

\begin{eg}\label{comm}
Let $\H$ be a separable Hilbert space. Fix an orthonormal basis $\{\xi_i\}$ for $\H$.
As usual, we represent operators on $\H$ as matrices with respect to $\{\xi_i\}$.
Let $\D$ be the norm closed subalgebra of $\A$ consisting of diagonal operators.
Then $\D$ is a Hilbert module over itself; let us notice that $\D$ is commutative.
Then $\langle x,y\rangle=0$ if and only if $x^*y=y^*x=yx^*=xy^*=0$.
Define
$f : \D \to \D$ by
$f(x)=x(x^*)^2.$
Then $f$ is an odd orthogonally additive mapping, but it is clearly not additive.
\end{eg}

\section{O.~A.~M.~on Hilbert $C^*$-modules over a $C^*$-algebra of compact operators and Hilbert $H^*$-modules}\label{C*&H*}

Let $\A$ be an arbitrary  $C^*$-algebra of compact operators.
By \cite[Theorem 1.4.5]{a}
$$\A=\oplus_{j \in J} \K(\H_j) = \{ (a_j) \in \Pi_{j \in J} \K(\H_j) \, : \, \lim_{j \in J} \Vert a_j \Vert = 0\}.$$
Let $W$ be a Hilbert $\A$-module. We may assume that $W$ is full.
If $W_j$ denotes the closed linear span of $W\K(\H_j),$ then each $W_j$ is a (full) Hilbert $\K(\H_j)$-module and
$W$ is the outer direct sum of $W_j$'s:
$$W=\oplus_{j \in J} W_j = \{ (w_j) \in \Pi_{j \in J} W_j \, : \, \lim_{j \in J} \Vert w_j \Vert = 0\}$$
(see \cite[Introduction]{bg2} or \cite{Sch}).

Now let $\A$ be an arbitrary $H^*$-algebra.
By \cite[Theorems 4.2 and 4.3]{ambrose} $\A$ is the orthogonal sum $\oplus_{j \in J} \A_j$ where each $\A_j$ is a simple $H^*$-algebra which is a minimal closed ideal of $\A$
and $\A_j=\H\S(\H_j)$ for some Hilbert space $\H_j.$ Then every $a \in \A$ can be written as $a=\sum_{j \in J}a_j$ with $a_j \in \H\S(\H_j)$ and $\Vert a \Vert^2=\sum_{j \in J} \Vert a_j \Vert^2.$
Let $W$ be a Hilbert $\A$-module. We may assume that $W$ is faithful (i.e.~it has zero annihilator in $\A$).
According to \cite[Theorem 2.3]{cmr} there exists a family $\{W_j \, : \, j \in J\}$ such that each $W_j$ is a (faithful) Hilbert $\H\S(\H_j)$-module and $W$ is the mixed product of $W_j$'s:
$$W=\times_{j \in J} W_j = \{ (w_j) \in \Pi_{j \in J} W_j \, : \, \sum_{j \in J} \Vert w_j \Vert^2 < \infty\}.$$

\begin{thm}\label{4.1}
Let $\A=\oplus_{j \in J} \A_j$ be a $C^*$-algebra of compact operators, resp.~an $H^*$-algebra, with $\A_j = \K(\H_j),$ resp.~$\A_j = \H\S(\H_j)$.
Let $W=\oplus_{j \in J} W_j$ be a Hilbert $\A$-module with $W_j$ a Hilbert $\A_j$-module such that $\dim_{\A_j} W_j = \dim \H_j=\aleph_0$ for each $j \in J.$
Let $G$ be a normed space and let $f : W \to G$ be a continuous o.~a.~m.
Then there exist a continuous additive mapping $T : W \to G$ and a continuous linear mapping $\Phi : \A \to G$ such that
$$f(x)=T(x)+\Phi(\langle x,x \rangle)\quad \text{for all}\quad x \in W.$$
\end{thm}

\begin{proof}
Define $f_j = f \vert_{W_j}$ for each $j \in J.$
Then $f_j : W_j \to G$ is a continuous o.~a.~m.
By Theorem \ref{3.1}, there exist a continuous additive mapping $T_j : W_j \to G$ and a continuous linear mapping $\Phi_j : \A_j \to G$ such that
$$
f_j(x_j)=T_j(x_j)+\Phi_j(\langle x_j, x_j \rangle)\quad \text{for all}\quad x_j \in W_j.
$$

Define $T : W \to G$ by
\[
T(x)=\frac{1}{2}\big(f(x)-f(-x)\big).
\]
 If we write $x=\sum_{j \in J} x_j$ with $x_j \in W_j,$ then
$$T(x)=\frac{1}{2}\sum_{j \in J}\big(f_j(x_j)-f_j(-x_j)\big)=\sum_{j \in J} T_j(x_j).$$
This implies that $T$ is an additive mapping; it is continuous since $f$ is continuous.

Let $\{w_i \, : \, i \in I\}$ be an orthonormal basis for $W$ and let $\{{w_i}^j \, : \, i \in I_j\} \subseteq \{w_i \, : \, i \in I\}$ be an orthonormal basis for $W_j.$
By Lemma \ref{bases}, without loss of generality we can assume $\langle {w_i}^j, {w_i}^j \rangle = {\xi_i}^j \otimes {\xi_i}^j$ where $\{ {\xi_i}^j \, : \, i \in I_j\}$ is an orthonormal basis for $\H_j.$
Let $a_j \in \A_j.$ Then
\begin{align*}
\Phi_j(a_j^*a_j)
&= \Phi_j(\sum_{i \in I_j} a_j^* ({\xi_i}^j \otimes {\xi_i}^j) a_j) \nonumber
   = \sum_{i \in I_j} \Phi_j (a_j^* \langle {w_i}^j, {w_i}^j \rangle a_j) \\
&= \sum_{i \in I_j} \Phi_j (\langle {w_i}^j a_j, {w_i}^j a_j\rangle) \nonumber
   = \Phi_j (\langle \sum_{i \in I_j} {w_i}^j a_j, \sum_{i \in I_j} {w_i}^j a_j\rangle). \nonumber
\end{align*}
Notice that $\sum_{i \in I} w_i a$ converges for all $a \in \A.$
In fact, if $a= \sum_{j \in J} a_j$ with $a_j \in \A_j,$ then Remark \ref{rem3.2} implies that $\sum_{i \in I_j} w_i^j a_j \in W_j$ and
$$\langle \sum_{i \in I_j} w_i^j a_j, \sum_{i \in I_j} w_i^j a_j \rangle = a_j^*a_j.$$
Then
\begin{align*}
\Vert \sum_{i \in I_j} w_i^j a_j \Vert_{W_j}
&= \Vert \langle \sum_{i \in I_j} w_i^j a_j, \sum_{i \in I_j} w_i^j a_j \rangle \Vert^{\frac{1}{2}} \nonumber \\
&= \Vert a_j^*a_j \Vert^{\frac{1}{2}} = \Vert a_j \Vert. \nonumber
\end{align*}
Hence $\sum_{j \in J} \sum_{i \in I_j} {w_i}^j a_j \in W,$ that is, $\sum_{i \in I} w_i a \in W.$
For $a \in \A$ define $\Phi(a^*a)=\frac{1}{2}\big(f(\sum_{i \in I} w_ia)+f(-\sum_{i \in I} w_ia)\big).$
If we write $a=\sum_{j \in J} a_j$ with $a_j \in \A_j,$ then
\begin{eqnarray}
\Phi(a^*a)
&=&\frac{1}{2} \sum_{j \in J} \big(f_j(\sum_{i \in I_j} {w_i}^ja_j)+f_j(-\sum_{i \in I_j}{w_i}^ja_j)\big) \nonumber \\
&=& \sum_{j \in J} \Phi_j (\langle \sum_{i \in I_j} {w_i}^ja_j, \sum_{i \in I_j} {w_i}^ja_j \rangle)\nonumber\\
&=& \sum_{j \in J} \Phi_j(a_j^*a_j). \nonumber
\end{eqnarray}
Using the fact that every $a \in A$ can be written as a linear combination of four positive elements (i.e., those of the form $x^*x$ for some $x \in \A$),
we define $\Phi(a)=\sum_{j \in J} \Phi_j(a_j)$ for all $a=\sum_{j \in J} a_j \in A.$
Since each $\Phi_j$ is linear, $\Phi$ is linear as well.
Since
$$\Vert \Phi(a^*a)\Vert \leq \Vert f \Vert \cdot \Vert \sum_{i \in I} w_ia \Vert = \Vert f \Vert \cdot \Vert a \Vert$$
for every $a \in \A,$ $\Phi$ is continuous.
Finally,
$$f(x)=T(x)+\Phi(\langle x, x \rangle)\quad \text{for all}\quad x \in W.$$
\end{proof}

\begin{cor}\label{compact}
Let $\A= \oplus_{j \in J}\A_j$ be a $C^*$-algebra of compact operators, resp.~an $H^*$-algebra, with $\A_j = \K(\H_j),$ resp.~$\A_j = \H\S(\H_j),$ such that $2 \leq \dim{\H_j} \leq \aleph_0$ for each $j \in J.$
Let $G$ be a normed space and let $f : \A \to G$ be a continuous o.~a.~m., with respect to the orthogonality defined by
$$x \perp y \Longleftrightarrow x^*y=0.$$
Then there exist a unique continuous additive mapping $T : \A \to G$ and a unique continuous linear mapping $\Phi : \A \to G$ such that
$$f(x)=T(x)+\Phi(x^*x)\quad \text{for all}\quad x \in \A.$$
\end{cor}

Let us mention that Example \ref{comm} also provides a counterexample for Corollary \ref{compact} in the case when $\A = W = \oplus_{j \in J} \A_j$ with $\A_j = W_j = \K(\H_j)$ or $\H\S(\H_j)$ and $\dim_{\A_j} W_j = \dim \H_j = 1$ for all $j \in J.$

\section{O.~A.~M.~on $\B(\H_1, \H_2)$}\label{B(H)}

The aim of this section is to prove an analogue of Corollary \ref{3.4} for $\B(\H_1, \H_2)$ instead of $\K(\H)$ and $\H\S(\H).$

\begin{prop}\label{5.1}
Let $\H_1$ and $\H_2$ be Hilbert spaces with $\dim{\H_1}, \dim{\H_2}\ge 2.$
Let $G$ be a uniquely 2--divisible abelian group and let $f : \B(\H_1, \H_2) \to G$ be an o.~a.~m.,
with respect to the orthogonality defined by
$$x \perp y \Longleftrightarrow x^*y=0.$$
Then there exist a unique additive mapping $T : \B(\H_1, \H_2) \to G$ and a symmetric biadditive orthogonality preserving mapping $B : \B(\H_1, \H_2) \times \B(\H_1, \H_2) \to G$ such that
$$f(x)=T(x)+B(x,x)\quad\text{for all}\quad x \in \B(\H_1, \H_2).$$

If $\H$ is a Hilbert space such that $\dim{\H} \ge 2$, $G$ is a Banach space and $f$ is continuous, then $T$ is continuous and there exists a unique continuous linear mapping
$\Phi : \B(\H) \to G$ such that
$$f(x)=T(x)+\Phi(x^*x)\quad\text{for all}\quad x \in \B(\H).$$
\end{prop}

\begin{proof}
Let us emphasize that $\B(\H_1, \H_2)$ is a Hilbert $\B(\H_1)$-module with respect to the inner product $\langle x, y \rangle = x^*y.$
If $\H_2$ is infinite dimensional, let $\K$ be a closed subspace of $\H_2$ such that both $\K$ and ${\K^\perp}$ are infinite dimensional.
If $\dim \H_2 = 2n,$ let $\K$ be an $n$-dimensional subspace of $\H_2.$
Let $U:\K\to \K^\perp$ be unitary. Then $\varphi: \B(\H_1, \K) \to \B({\H_1, \K^\perp}),$ $\varphi(A)=UA$  is an isomorphism.
Notice that $\B(\H_1, \K) \oplus \varphi(\B(\H_1, \K))  = \B(\H_1, \H_2).$
It remains to apply Theorem \ref{2.1} (i).

If $\dim \H_2 =2n+1,$ let $\K$ be a $1$-dimensional subspace of $\H_2.$
Then $\dim {\K^\perp} = 2n$ and according to the above, the statement holds on $\B(\H_1, \K^{\perp}).$
Let $\M$ be a $2$-dimensional subspace of $\H_2$ containing $\K.$
Again, according to the above, the statement also holds true on $\B(\H_1, \M),$ hence
finally on $\B(\H_1, \H_2).$

The second statement can be proved in a similar way, but using Theorem 3.1 (ii) instead of Theorem 3.1 (i), and then applying the results from [3] (see also [4, Theorem 1.1]) to reperesent $ S $ via $ \Phi $.Ò
\end{proof}

\end{document}